\documentclass[12pt]{amsart}
\usepackage{amssymb, a4wide, mathdots,url,hyperref}
\usepackage{bbm}

\renewcommand{\subset}{\subseteq}

\newtheorem{theorem}{Theorem}[section]
\newtheorem{lemma}[theorem]{Lemma}
\newtheorem{proposition}[theorem]{Proposition}
\newtheorem{remark}[theorem]{Remark}
\newtheorem{conjecture}[theorem]{Conjecture}
\newtheorem{definition}[theorem]{Definition}
\newtheorem{corollary}[theorem]{Corollary}

\author{Maxim Gurevich}
\address{Department of Mathematics, Weizmann Institute of Science, Rehovot 7610001, Israel}
\email{max.gurevich@weizmann.ac.il}
\date{\today}

\newcommand{\gotM}{\mathfrak{m}}
\newcommand{\gotN}{\mathfrak{n}}

\newcommand{\wid}{\omega}

\DeclareMathOperator{\irr}{Irr}
\DeclareMathOperator{\seg}{Seg}
\DeclareMathOperator{\supp}{supp}

\begin{document}

\title{A filtration on rings of representations of non-archimedean $GL_n$}

\begin{abstract}
Let $F$ be a $p$-adic field. Let $\mathcal{R}$ be the Grothendieck ring of complex smooth finite-length representations of the groups $\{GL_n(F)\}_{n=0}^\infty$ taken together, with multiplication defined in the sense of parabolic induction. We introduce a width invariant for elements of $\mathcal{R}$ and show that it gives an increasing filtration on the ring. Irreducible representations of width $1$ are precisely those known as ladder representations. 

We thus obtain a necessary condition on irreducible factors of a product of two ladder representations. For such a product we further establish a multiplicity-one phenomenon, which was previously observed in special cases.

\end{abstract}

\maketitle

\section{Introduction}
Let $F$ be a $p$-adic field. Let $\mathcal{R}_n$ be the Grothendieck group associated with the category of complex smooth finite-length representations of the group $GL_n(F)$. Given two smooth representations $\pi_i$ of $GL_{n_i}$ ($i=1,2$), $\pi_1\times \pi_2$ is defined as the parabolic induction of $\pi_1\otimes \pi_2$ to $GL_{n_1+n_2}(F)$ from a block upper-triangular maximal Levi subgroup. This product operation equips the group $\mathcal{R} = \oplus_{n\geq 0} \mathcal{R}_n$ with a structure of a commutative ring.

This ring is long known (see \cite{Zel}) to be a polynomial ring over $\mathbb{Z}$ in infinitely many variables. One way to observe this is by recalling the Langlands classification, which gives a bijection from the irreducible representations $\irr = \cup_{n\geq 0} \irr(GL_n(F))$ to the so-called standard representations. The collection of standard representations is closed under multiplication and gives a basis to $\mathcal{R}$ as a free abelian group. In particular, the essentially square-integrable (segment representations) elements of $\irr$ freely generate $\mathcal{R}$ as a polynomial ring over $\mathbb{Z}$.

Yet, the collection $\irr$ itself gives a different basis to $\mathcal{R}$ as a free abelian group. Our note joins an effort to describe the multiplicative structure of $\mathcal{R}$ in terms of this natural basis. 

It is known (see \cite{Hender} for an efficient description) that the transition matrix between the two mentioned above bases, i.e.~standard and irreducible representations, is given by values of Kazhdan-Lusztig polynomials for symmetric groups. Thus, given $\pi_1,\pi_2\in \irr$, the irreducible factors of $\pi_1\times \pi_2$ can in principle be determined by computing these polynomials. 

However, there is hope that the complexity involved in such a computation may be overcome by applying more direct methods on the problem. For example, irreducibility of $\pi_1\times \pi_2$ for unitarizable representations is a classical result due to Bernstein \cite{bern}. The results of Lapid and M\'{i}nguez \cite{LM2} further deal with this question and supply direct combinatorial criteria for irreducibility in some cases.

One class of irreducible representations which was shown to be susceptible to such methods is that of ladder representations, introduced in \cite{LM}. In \cite{dandan}, an interpretation of the role ladder representations as analogous to that of finite-dimensional representations in category $\mathcal{O}$ was given. In this work we present a new invariant of finite-length representations which attempts to quantify the distance from the well-behaved properties of ladder representations.

Given a supercuspidal $\rho\in \irr$, we write $\irr_{[\rho]}$ for the set of elements of $\irr$ whose supercuspidal support consists of representations on the supercuspidal line $\{\rho\otimes |\det|^n\}_{n\in\mathbb{Z}}$. Let $\mathcal{R}_{[\rho]}$ be the group (which is a subring) generated by $\irr_{[\rho]}$ in $\mathcal{R}$. It is enough to study the multiplicative structure of $\mathcal{R}_{[\rho]}$ 
for a fixed $\rho$, since any $\pi\in\irr$ is uniquely decomposed as $\pi= \pi_1\times\cdots\times \pi_k$, where $\pi_i \in \irr_{\rho_i}$, for $\{\rho_i\}$ that belong to disjoint supercuspidal lines.

Recall that the Zelevinski classification\footnote{ In what follows we work with the Langlands classification, but we refer to it in terms of Zelevinski's multisegments. The classifications are dual to each other, as explained for example in \cite{LM2}.} \cite{Zel} of $\irr_{[\rho]}$ describes each irreducible representation as a multiset of segments, i.e.~intervals of the form $[a,b],\,a,b\in\mathbb{Z}$. A ladder representation would then be given by a set of the form $\{[a_i,b_i]\}_{i=1}^k$, with $a_1<\ldots<a_k$ and $b_1<\ldots<b_k$.

For any $\pi\in \irr_{[\rho]}$, we call the \textit{width} $\wid(\pi)$ of $\pi$ to be the minimal number of ladder sets of segments required to cover disjointly the multiset attached to $\pi$. Our first result claims that $\wid$ serves as a length function for an increasing filtration on the ring $\mathcal{R}_{[\rho]}$ in the following sense.

\begin{theorem}\label{main-1}
Let $\rho\in \irr$ be a supercuspidal representation. For $n\in\mathbb{N}$, let $\mathcal{M}_n$ be the subgroup generated in $\mathcal{R}_{[\rho]}$ by all $\pi\in \irr_{[\rho]}$ with $\wid(\pi)\leq n$. Then,
\[
\mathcal{M}_i\cdot \mathcal{M}_j\subset \mathcal{M}_{i+j},\qquad\forall i,j\in\mathbb{N}.
\]
\end{theorem}

In particular, Theorem \ref{main-1} gives for any $\pi_1,\pi_2\in \irr$ a necessary condition for the occurrence of irreducible representations in the composition series of $\pi_1\times \pi_2$. 

The problem of finding a general rule in terms of multisegments for the precise composition series of $\pi_1\times\pi_2$ appears to be far from reach (as mentioned, even determining irreducibility proved to be difficult). Nevertheless, on the first level of our filtration, namely when $\pi_1,\pi_2$ are ladder representations, this task was shown to be feasible by the works of Tadic \cite{tadic-speh} and Leclerc \cite{lecl}. These provide such formulas for a product of ladder representations taken from certain subclasses. 

Our second result establishes a general phenomenon previously observed for those subclasses. 

\begin{theorem}\label{main-2}
For any ladder representations $\pi_1\in \irr(G_{n_1}),\,\pi_2\in \irr(G_{n_2})$, the isomorphism classes of irreducible subquotients of $\pi_1\times\pi_2$ are all of multiplicity one. 

In other words, in the group $\mathcal{R}_{n_1+n_2}$, we have $[\pi_1 \times \pi_2] = [\sigma_1]\, + \,\ldots +\, [\sigma_k]$, for $\sigma_1,\ldots,\sigma_k\in \irr(G_{n_1+n_2})$, with $\sigma_i\ncong \sigma_j$ for $i\neq j$.

\end{theorem}

Both Theorems \ref{main-1} and \ref{main-2} provide ``upper bounds" on the set of irreducible subquotients of a parabolically induced representation. Let us finish by stating a conjecture of Erez Lapid, which claims that in a certain sense these bounds are tight for the case of a product of ladder representations. 
\begin{conjecture}\label{conje}
Let $a_1 < \ldots < a_{2k} < b_1 < \ldots < b_{2k}$ be given integers. Suppose that $\pi\in \irr_{[\rho]}$ is the ladder representation associated with the multiset $\{[a_{2i},b_{2i}]\}_{i=1}^k$, $\pi'\in \irr_{[\rho]}$ is the ladder representation associated with the multiset $\{[a_{2i-1},b_{2i-1}]\}_{i=1}^k$ and $\lambda$ is the standard representation associated with the union multiset $\{[a_j,b_j]\}_{j=1}^{2k}$. Then, in $\mathcal{R}$, we have

\[
[\pi\times\pi'] = \sum_{\begin{array}{cc}\scriptstyle \sigma\in \irr_{[\rho]} : \: \scriptstyle  \wid(\sigma)\leq 2  ,\;\\ \scriptstyle\sigma \mbox{ \footnotesize is a subquotient of } \lambda \end{array}} [\sigma]\;.
\]
\end{conjecture}
In other words, there is an expected case in which all possible representations of width $2$ should occur in a single product. We expect this case to provide insight towards a formulation of a general rule for the irreducible factors of a product of ladder representations.


\subsection*{Acknowledgements}
I would like to thank Erez Lapid for introducing the subject to me and for many fruitful discussions.

\section{Notation and preliminaries}
\subsection{Generalities}
For a $p$-adic group $G$, let $\mathfrak{R}(G)$ be the category
of smooth complex representations of $G$ of finite length. Denote by $\irr(G)$ the set of equivalence classes
of irreducible objects in $\mathfrak{R}(G)$. Denote by $\mathcal{C}(G)\subset \irr(G)$ the subset of irreducible supercuspidal representations. Let $\mathcal{R}(G)$ be the Grothendieck group of $\mathfrak{R}(G)$. We write $\pi\mapsto [\pi]$ for the canonical map $\mathfrak{R}(G) \to \mathcal{R}(G)$.

Given $\pi\in \mathfrak{R}(G)$, we have $[\pi] = \sum_{\sigma\in \irr(G)} c_\sigma\cdot [\sigma]$. For every $\sigma\in \irr(G)$, let us denote the multiplicity $m(\sigma, \pi):= c_{\sigma}\geq0$.

Now let $F$ be a fixed $p$-adic field. We write $G_n = GL_n(F)$, for all $n\geq1$, and $G_0$ for the trivial group.

For a given $n$, let $\alpha = (n_1, \ldots, n_r)$ be a composition of $n$. We denote by $M_\alpha$ the subgroup of $G_n$ isomorphic to $G_{n_1} \times \cdots \times G_{n_r}$ consisting of matrices which are diagonal by blocks of size $n_1, \ldots, n_r$ and by $P_\alpha$ the subgroup of $G_n$ generated by $M_\alpha$ and the upper
unitriangular matrices. A standard parabolic subgroup of $G_n$ is a subgroup of the form $P_\alpha$ and its standard Levi factor is $M_\alpha$. We write $\mathbf{r}_\alpha: \mathfrak{R}(G_n)\to \mathfrak{R}(M_\alpha)$ and $\mathbf{i}_\alpha: \mathfrak{R}(M_\alpha)\to \mathfrak{R}(G_n)$ for the normalized Jacquet functor and the parabolic induction functor associated to $P_\alpha$.

Note that naturally $\mathcal{R}(M_\alpha)\cong \mathcal{R}(G_{n_1})\otimes \cdots\otimes \mathcal{R}(G_{n_r})$ and $\irr(M_\alpha)=\irr(G_{n_1})\times\cdots\times \irr(G_{n_r})$. 
\begin{definition}
We say that a representation $\sigma\in \irr(G_m)$ is a \textit{Jacquet module component} of $\pi\in \mathfrak{R}(G_n)$ if there is a Levi subgroup $M_\alpha< G_n$ and a representation $\tau=\tau_1\otimes\cdots\otimes \tau_r\in \irr(M_\alpha)$, such that $\sigma\cong \tau_i$ for some $i$ and $m(\tau,\mathbf{r}_\alpha(\pi))>0$.
\end{definition}

For $\pi_i\in \mathfrak{R}(G_{n_i})$, $i=1,\ldots,r$, we write
\[
\pi_1\times\cdots\times \pi_r = \mathbf{i}_{(n_1,\ldots,n_r)}(\pi_1\otimes\cdots\otimes \pi_r)\in \mathfrak{R}(G_{n_1+\ldots+n_r}).
\]
Let us write $\mathcal{R} = \oplus_{m \geq 0} \mathcal{R}(G_m)$. This product operation defines a commutative ring structure on the group $\mathcal{R}$, where the trivial one-dimensional representation of $G_0$ is treated as an identity element. 

We also write $\irr = \cup_{m\geq0} \irr(G_m)$ and $\mathcal{C} = \cup_{m\geq1} \mathcal{C}(G_m)$.

Given a set $X$, we write $\mathbb{N}(X)$ for the commutative semigroup of maps from $X$ to $\mathbb{N}= \mathbb{Z}_{\geq0}$ with finite support. For $A\in\mathbb{N}(X)$, we write
\[
\underline{A} = \{x\in X \,:\, A(x)>0\}\subset X
\]
for the support of $A$. 
Given a set $S\subset X$, we write $\mathbbm{1}_S\in \mathbb{N}(X)$ for the indicator function of $S$. This gives an embedding $X\to \mathbb{N}(X)$ by $x \mapsto \mathbbm{1}_x$. We will sometimes simply refer to $X$ as a subset of $\mathbb{N}(X)$ by implicitly using this embedding.

For $A\in \mathbb{N}(X)$, we write $\# A= \sum_{x\in X} A(x)$ for the size of $A$.

For $A,B\in \mathbb{N}(X)$ we say that $A\leq B$ if $B-A\in \mathbb{N}(X)$.

\subsection{Langlands classification}
Let us describe the Langlands classification of $\irr$ in terms convenient for our needs. 

For any $n$, let $\nu^s= |\det|^s_F,\;s\in \mathbb{C}$ denote the family of one-dimensional representations of $G_n$, where $|\cdot|_F$ is the absolute value of $F$. For $\pi\in \mathfrak{R}(G_n)$, we write $\pi\nu^s := \pi\otimes \nu^s\in \mathfrak{R}(G_n)$.

Given $\rho\in \mathcal{C}(G_n)$ and two integers $a\leq b$, we write $L([a,b]_\rho)\in \irr(G_{n(b-a+1)})$ for the unique irreducible quotient of $\rho\nu^{a}\times \rho\nu^{a+1}\times\cdots\times \rho\nu^b$. It will also be helpful to set $L([a,a-1]_\rho)$ as the trivial representation of $G_0$.

We also treat the \textit{segment} $\Delta= [a,b]_\rho$ as a formal object defined by the triple $([\rho],a,b)$. We denote by $\seg$ the collection of all segments that are defined by $\rho\in \mathcal{C}$ and integers $a-1\leq b$, up to the equivalence $[a,b]_\rho=[a',b']_{\rho'}$, when $\rho\nu^a\cong \rho'\nu^{a'}$ and $\rho\nu^b \cong \rho'\nu^{b'}$.

A segment $\Delta_1$ is said to precede a segment $\Delta_2$, if $\Delta_1 = [a_1,b_1]_{\rho} ,\;\Delta_2= [a_2,b_2]_{\rho}$ and $a_1\leq a_2-1\leq b_1<b_2$. We will write $\Delta_1 \prec \Delta_2$ in this case and say that the pair $\{\Delta_1,\Delta_2\}$ is linked.

We will write $[a_1,b_1]_{\rho}\subseteq[a_2,b_2]_{\rho}$ when $a_2\leq a_1$ and $b_1\leq b_2$.

The elements of $\mathbb{N}(\seg)$ are called \textit{multisegments}. Langlands classification gives a bijection 
\[
L: \mathbb{N}(\seg) \to \irr
\]
that extends the definition of $L$ for a single segment described above.

Given a non-zero multisegment $\gotM$, it is possible to write it as $\gotM = \Delta_1+\ldots+\Delta_k$, where $\Delta_i\in \seg$ for all $i$, and $\Delta_j \nprec \Delta_i$ for all $i<j$. We then define the \textit{co-standard module} associated with $\gotM$ to be the representation
\[
\tilde{\lambda}(\gotM) = L(\Delta_1)\times\cdots\times L(\Delta_k)\;.
\]
The isomorphism class of $\tilde{\lambda}(\gotM)$ does not depend on the enumeration of segments, as long as the condition above is satisfied. The representation $\tilde{\lambda}(\gotM)$ has a unique irreducible sub-representation, which is isomorphic to $L(\gotM)$. We refer to \cite{LM2} for a more thorough discussion of the classification with a similar terminology.

When none of the pairs of segments in $\gotM$ are linked the representation $L(\gotM)$ is called generic. In that case $\tilde{\lambda}(\gotM)\cong L(\gotM)$.

\begin{remark}\label{rmk}
We always have $m(L(\gotM_1+\gotM_2), L(\gotM_1)\times L(\gotM_2))=1$, for $\gotM_1,\gotM_2\in \mathbb{N}(\seg)$.
\end{remark}

\subsection{Supercuspidal lines}
For every $\pi\in \irr$ there exist $\rho_1,\ldots,\rho_r \in\mathcal{C}$ for which $\pi$ is a sub-representation of $\rho_1\times\cdots\times \rho_r$. The notion of supercuspidal support can then be defined as 
\[
\supp(\pi):= \rho_1+\ldots+\rho_r\in \mathbb{N}(\mathcal{C})\;.
\]
Note that supercuspidal supports can be easily read from the Langlands classification. Namely, for all $\gotM\in \mathbb{N}(\seg)$ and all $\rho\in \mathcal{C}$, $\supp(L(\gotM))(\rho) = 
\sum_{\Delta\in \seg\,:\, \rho\in \underline{\supp (\Delta)} }\gotM(\Delta)$, while for a single segment $\Delta=[a,b]_\rho$ we have $\supp(\Delta) = \rho\nu^a + \rho\nu^{a+1}+\cdots+\rho\nu^b$.

Given $\rho\in \mathcal{C}$, we call 
\[
\mathbb{Z}_{[\rho]}:=\{\rho\nu^a\,:\;a\in \mathbb{Z}\}\subset \mathcal{C}
\]
the \textit{line} of $\rho$. 

We write $\irr_{[\rho]}\subset \irr$ for the collection of irreducible representations whose supercuspidal support is supported on $\mathbb{Z}_{[\rho]}$. We also write $\seg_{[\rho]} = \{[a,b]_\rho\in \seg\,:\;a-1\leq b\}$ and $\mathcal{R}_{[\rho]}$ for the ring generated by $\irr_{[\rho]}$ in $\mathcal{R}$. It is then straightforward that the restriction of $L$ gives a bijection $\mathbb{N}(\seg_{[\rho]})\to \irr_{[\rho]}$.
 
If the lines of $\rho_1,\ldots , \rho_r\in\mathcal{C}$ are distinct and $\pi_i \in \irr_{\rho_i}$, then $\pi_1 \times \cdots \times \pi_r$ is irreducible. Thus, we can deal with questions of decomposition of induced representations in $\mathcal{R}$ by analyzing $\mathcal{R}_{[\rho]}$, for a single $\rho\in \mathcal{C}$.

\section{Width invariant}

Recall that $\pi\in \irr$ is called a \textit{ladder representation} if $\pi\in \irr_{[\rho]}$ for some $\rho\in \mathcal{C}$ and $\pi= L(\gotM)$, where $\gotM=[a_1,b_1]_\rho+\ldots+[a_k,b_k]_\rho\in \mathbb{N}(\seg_{[\rho]})$ is such that $a_1<\ldots<a_k$ and $b_1<\ldots<b_k$. In this case we will also call $\gotM$ a ladder multisegment.

Let us fix a supercuspidal representation $\rho\in\mathcal{C}$ for the rest of this note. 
We will naturally identify $\mathbb{Z}_{[\rho]}$ with $\mathbb{Z}$. For $\pi\in \irr_{[\rho]}$ we will then refer to $\supp(\pi)$ as an element of $\mathbb{N}(\mathbb{Z})$. A non-trivial $\Delta\in \seg_{[\rho]}$ can be uniquely written as $\Delta = [a,b]_\rho$. Therefore, we will simply write $\Delta=[a,b]$. We then write $b(\Delta)=a$ and $e(\Delta)=b$.

\begin{definition}
The \textit{width} $\wid(\gotM)$ of a non-trivial multisegment $\gotM\in \mathbb{N}(\seg_{[\rho]})$ is the minimal number $k$, for which it is possible to write $\gotM = \gotM_1 +\ldots + \gotM_k$ for some ladder multisegments $\gotM_1,\ldots,\gotM_k\in \mathbb{N}(\seg_{[\rho]}) $. 

We also write $\wid(\pi) = \wid(\gotM)$ for the width of the representation $\pi = L(\gotM)\in \irr_{[\rho]}$.

For any $\pi\in \mathfrak{R}(G_n)$ with $[\pi]\in \mathcal{R}_{[\rho]}$ the definition is extended by
\[
\wid(\pi) = \max\{ \wid(\sigma)\::\:\sigma\in\irr(G_n),\; m(\sigma,\pi)>0\}.
\]

\end{definition}

Note that ladder representations are precisely those irreducible representations with width $1$. Note too that $\wid(\gotM)$ is always bounded by the number $\# \gotM$ of segments in $\gotM$.

Now, let us consider the relation $\preceq'$ on $\seg_{[\rho]}$ that is defined by
\[
[a,b]\preceq' [c,d]\quad \Leftrightarrow \quad \mbox{either}\;\left\{\begin{array}{l} a< c\\ b<d\end{array} \right. \;\mbox{or}\; [a,b]=[c,d].
\]
This relation can be viewed as the transitive and reflexive closure of $\prec$. Note that if $\Delta_1,\Delta_2\in \seg_{[\rho]}$ are such that $\Delta_1\npreceq' \Delta_2$ and $\Delta_2\npreceq' \Delta_1$, then we must have either $\Delta_1\subseteq \Delta_2$ or $\Delta_2 \subseteq \Delta_1$.

\begin{lemma}\label{dilw}
For every $\gotM\in \mathbb{N}(\seg_{[\rho]})$,
\[
\wid(\gotM) = \max\{ k\;:\;  \mbox{there are non-trivial segments }\Delta_1\subseteq \ldots\subseteq \Delta_k\mbox{ s.t. }\Delta_1+\ldots+\Delta_k\leq\gotM \}.
\]
\end{lemma}

\begin{proof}
Note that the collection of segments in $\gotM$, counted with multiplicities, together with the relation $\preceq'$ is a poset. A chain for this poset would give a ladder multisegment, while an antichain is a multisegment $\Delta_1+\ldots+\Delta_n\leq \gotM$ for which $\Delta_1\subseteq \ldots\subseteq \Delta_n$ holds. Thus, the statement follows from Dilworth's theorem \cite{dilw}.

\end{proof}

It follows from the Geometric Lemma of Bernstein-Zelevinski (see \cite[Section 1.2]{LM2}) that the collection of Jacquet module components of $\pi_1\times\cdots\times \pi_k\in \mathfrak{R}(G_n)$ is precisely the collection of representations $\tau\in \irr$ for which $m(\tau, \sigma_1\times\cdots\times \sigma_k)>0$ holds, for some choice of Jacquet module components $\sigma_i$ of $\pi_i$, for $i=1,\ldots,k$.

For $\sigma\in \irr_{[\rho]}$, let $b(\sigma)\in \mathbb{Z}_{[\rho]}\cong \mathbb{Z}$ be the minimal element in $\underline{\supp(\sigma)}$. We write $B(\sigma)=\supp(\sigma)(b(\sigma))$ for the multiplicity of $b(\sigma)$ in $\supp(\sigma)$.

For $\pi\in \mathfrak{R}(G_n)$  with $[\pi]\in \mathcal{R}_{[\rho]}$, we write
\[
j(\pi) = \max \;\{\; B(\pi')\;:\; \pi'\in \irr\mbox{ is a Jacquet module component of }\pi\}.
\]

\begin{lemma}\label{first-ineq}
For every $\pi\in\mathfrak{R}(G_n)$ with $[\pi]\in \mathcal{R}_{[\rho]}$, we have
\[
\wid(\pi) \leq j(\pi).
\]
\end{lemma}
\begin{proof}
From exactness of the Jacquet functor, it suffices to prove the statement for $\pi\in\irr_{[\rho]}$.

Let $\pi\in \irr_{[\rho]}$ be given. We write $\pi = L(\gotM)$ and $k=\wid(\pi)$. By Lemma \ref{dilw}, there are non-trivial segments $\Delta_1+\ldots+ \Delta_{k}\leq \gotM$ for which $\Delta_1\subseteq \ldots\subseteq \Delta_{k}$.
We write
\[
S = \left\{\Delta\in\seg_{[\rho]}\;:\; \Delta_i\preceq' \Delta\;\,\mbox{for some }1\leq i\leq k \right\}.
\]
Let us define the following multisegments:
\[\gotM_1 = \mathbbm{1}_S\cdot\left(\gotM-(\Delta_1+\ldots + \Delta_{k})\right),\]
\[\gotM_2= (1-\mathbbm{1}_S)\cdot\left(\gotM-(\Delta_1+\ldots + \Delta_{k})\right)= \gotM - \gotM_1 - (\Delta_1+\ldots + \Delta_{k}).\]

We claim that 
\[\tilde{\lambda}(\gotM)\cong  \tilde{\lambda}(\gotM_2)\times L(\Delta_1)\times \cdots \times L(\Delta_{k})\times \tilde{\lambda}(\gotM_1).\]
It is enough to check that $\Delta \nprec \Delta'$, for all $\Delta\in \underline{\gotM-\gotM_2}$ and all $\Delta'\in \underline{\gotM-\gotM_1}$. 

Assume the contrary for certain $\Delta,\Delta'$. Then $\Delta_{i_0}\preceq' \Delta$ for a certain $i_0$ and from the transitivity of $\preceq'$ we get $\Delta_{i_0}\preceq' \Delta'$. Hence, $\Delta'\not\in \underline{\gotM_2}$ and we must have $\Delta' = \Delta_j$ for some $j$. But, from $\Delta_{i_0}\preceq' \Delta \prec \Delta'= \Delta_j$ we get $\Delta = \Delta' = \Delta_{i_0}$, which is a contradiction to the non-reflexivity of $\prec$.

Since $\pi$ is embedded in $\tilde{\lambda}(\gotM)$, from Frobenius reciprocity we see that the representation $\pi':=L(\Delta_1+\ldots +\Delta_{k})\cong L(\Delta_1)\times \cdots \times L(\Delta_{k})$ is a Jacquet module component of $\pi$.

Let us write $\Delta_i = [a_i,b_i]$ for all $1\leq i\leq k$. Then, $a_1\leq \ldots\leq a_{k}$ and $b_1\geq \ldots \geq b_{k}$. From the formula for Jacquet modules of segment representations and the Geometric Lemma, we know that
\[
\mathbf{r}_\alpha(\pi') = L\left(\sum_{i=1}^{k}[a_{k},b_i]\right) \otimes L\left( \sum_{i=1}^{k} [a_i,a_{k}-1]\right),
\]
where $P_\alpha$ is the appropriate parabolic subgroup.

Thus, $\pi'':=  L\left(\sum_{i=1}^{k}[a_{k},b_i]\right)$ is a Jacquet module component of $\pi'$, and therefore also of $\pi$. Thus, $\wid(\pi) =k= B(\pi'')\leq j(\pi)$. 

\end{proof}

\begin{proposition}\label{main-prop}
Let $\pi_1,\ldots, \pi_k\in \irr_{[\rho]}$ be ladder representations. 

Then $\wid(\pi_1\times\cdots\times \pi_k)\leq k$.
\end{proposition}
\begin{proof}
A Jacquet module component $\sigma$ of $\pi_1\times\cdots\times \pi_k$ is a subquotient of $\sigma_1\times\cdots\times \sigma_k$, where each $\sigma_i$ is a Jacquet module component of $\pi_i$. By the result of \cite{LapidKret}, every such $\sigma_i$ must be a ladder representation, which means that $B(\sigma_i)=1$. It easily follows that $B(\sigma) = B(\sigma_1\times\cdots\times \sigma_k)\leq k$. Hence, $j(\pi_1\times\cdots\times \pi_k)\leq k$ and the result follows from Lemma \ref{first-ineq}.

\end{proof}

\begin{corollary}
For every $\pi\in\mathfrak{R}(G_n)$ with $[\pi]\in \mathcal{R}_{[\rho]}$,
\[
\wid(\pi) = j(\pi).
\]
\end{corollary}
\begin{proof}
It is enough to consider $\pi\in \irr$. By definition we can write $\pi = L(\gotM_1+\ldots+\gotM_k)$, where $k=\wid(\pi)$ and $\pi_i = L(\gotM_i)$ are ladder representations, for all $1\leq i\leq k$. Thus, $m(\pi,\:\pi_1\times\cdots\times \pi_k)>0$.
From exactness of the Jacquet functor and the proof of Proposition \ref{main-prop}, we get $j(\pi)\leq j(\pi_1\times\cdots\times \pi_k)\leq k$. Combining with Lemma \ref{first-ineq}, the statement now follows.
 
\end{proof}

Theorem \ref{main-1} now follows from the next corollary.

\begin{corollary}
For all $\pi_1\in\mathfrak{R}(G_{n_1}),\;\pi_2\in \mathfrak{R}(G_{n_2})$ with $[\pi_1],[\pi_2]\in \mathcal{R}_{[\rho]}$,
\[
\wid(\pi_1\times\pi_2)\leq \wid(\pi_1)+\wid(\pi_2).
\]
\end{corollary}

\begin{proof}
By exactness of parabolic induction it is enough to assume that $\pi_1,\pi_2\in \irr_{[\rho]}$. Suppose that $\pi_1 = L(\gotM_1+\ldots + \gotM_s)$ and $\pi_2=L(\gotM_{s+1}+\ldots+\gotM_t)$ for ladder representations $\tau_i=L(\gotM_i)$, where $1\leq i\leq t$ and $s=\wid(\pi_1),\: t-s = \wid(\pi_2)$. Then from Remark \ref{rmk} it follows that $m(\pi_1,\, \tau_1\times \cdots\times \tau_s)>0$ and $m(\pi_2,\,\tau_{s+1}\times\cdots\times \tau_t)>0$. Thus, again by exactness of parabolic induction $\pi_1\times\pi_2$ appears as a subquotient of $(\tau_1\times\cdots\times \tau_s)\times (\tau_{s+1}\times \cdots\times \tau_t)$. The statement now follows from Proposition \ref{main-prop}.

\end{proof}

\section{Multiplicity One}

\begin{lemma}\label{mul-lem}
Let $a,b,c$ be integers with $a < b$ and $a-1\leq c \leq  b$. Fix the representation $\pi =L([a,b]+[a,c])\in \irr_{[\rho]}$.

Suppose that $\pi_1=L(\gotM_1), \pi_2=L(\gotM_2)$ are ladder representations. Then,
\[
m(\pi,\:\pi_1\times \pi_2)\leq 1\;.
\] 

In case $m(\pi,\:\pi_1\times \pi_2)= 1$ holds, both $\pi_1,\pi_2$ must be generic representations.
Furthermore, when $c=b$ we must have $\pi_1\cong\pi_2\cong L([a,b])$. Otherwise, when $c<b$, the pair $\{\gotM_1,\gotM_2\}$ must be of the form
\[
\left\{  \sum_{i=0}^t [a_{2i},a_{2i+1}-1] ,\;  [a,c]+\sum_{i=1}^{s} [a_{2i-1},a_{2i}-1]   \right\},
\]
for some $a=a_0,\; c+1<a_1<\ldots<a_l=b+1$, with either $l=2t+1=2s+1$ or $l=2s= 2t+2$.

\end{lemma}

\begin{proof}
Since the pair of segments $\{[a,b], [a,c]\}$ is not linked, $\pi$ is a generic representation. Recall that means that the space of $\pi$ carries a non-zero Whittaker functional. From exactness of the Whittaker functor, we deduce that $\pi_1\times \pi_2$ carries a non-zero Whittaker functional as well. Now, by Rodier's theorem \cite{rodier} $\pi_1\otimes \pi_2$ is generic, hence, so are $\pi_1,\pi_2$. Moreover, since $\pi_1,\pi_2$ are irreducible it follows from the same theorem that the space of Whittaker functionals on $\pi_1\times \pi_2$ is one-dimensional. Again, by exactness this means $\pi$ cannot appear with multiplicity $>1$ in the product.

Note, that since $\gotM_i,\,i=1,2$ are multisegments for which $L(\gotM_i)$ are generic ladder representations, their segments must be pairwise unlinked. Thus, we can write $\gotM_i = [t^i_0,t^i_1-1] + [t^i_2,t^i_3-1] + \ldots + [t^i_{2k_i},t^i_{2k_i+1}-1]$, for $t^i_0<t^i_1<\ldots <t^i_{2k_i+1},\;i=1,2$. 

Note that,
\[
2\supp(L([a,c])) \leq \supp(L([a,c])+\supp(L[a,b]) = \supp(\pi) =
\]
\[= \supp(\pi_1\times \pi_2) = \supp(\pi_1)+\supp(\pi_2)\;.
\] 
Clearly it follows that $[a,c]\subseteq [t^1_{2j_1}, t^1_{2j_1+1}-1] $ and $[a,c]\subseteq [t^2_{2j_2}, t^2_{2j_2+1}-1]$ for some $j_1,j_2$.

The rest of the statement easily follows after noting that each of the supercuspidals appearing in $\underline{\supp(L([c+1,b]))}$ must appear only once in $\supp(\pi_1)+\supp(\pi_2)$.

\end{proof}

Given $\gotM\in \mathbb{N}(\seg_{[\rho]})$ with $\pi=L(\gotM)\in \irr(G_n)$, let us denote the collection of integers
\[
B_\gotM = \{b(\Delta)\::\:\Delta\in \underline{\gotM}\}\subset \mathbb{Z}_{[\rho]}\cong \mathbb{Z}\;.
\] 
We can write $B_\gotM = \{b_1,\ldots, b_k\}$ with $b_1<\ldots< b_k$. For each $1\leq i\leq k$, write $S_j = \{\Delta\in \seg_{[\rho]}\::\: b(\Delta)=b_j\}$. With these notations we can write 

\[
\pi_\otimes = L\left(\gotM\cdot\mathbbm{1}_{S_1}\right)\otimes \cdots\otimes L\left(\gotM\cdot\mathbbm{1}_{S_k}\right)
\]
for the representation of the corresponding Levi subgroup $M_{\alpha_\pi}$ of $G_n$. 

Since $L(\gotM\cdot\mathbbm{1}_{S_j})$ are generic representations, it is clear that $\tilde{\lambda}(\gotM)\cong \mathbf{i}_{\alpha_\pi}(\pi_\otimes)$. Since $\pi$ is embedded in $\tilde{\lambda}(\gotM)$, by adjunction we always find $\pi_\otimes$ as a quotient of $\mathbf{r}_{\alpha_\pi}(\pi)$.

\begin{proposition}\label{mult-one}
Let $\pi_1=L(\gotM_1),\pi_2=L(\gotM_2)$ be ladder representations in $\irr_{[\rho]}$. Suppose that $\pi_1\times\pi_2\in \mathfrak{R}(G_n)$. 

For any $\sigma\in \irr(G_n)$, we have $m(\sigma_\otimes,\,\mathbf{r}_{\alpha_\sigma} (\pi_1\times\pi_2))\leq1$.
\end{proposition}
\begin{proof}
Let $\sigma = L(\gotN)\in \irr(G_n)$. We will prove the statement by induction on the number $\# \gotN$ of segments in $\gotN$.

Let $\Delta\in\underline{\gotN}$ be a segment with minimal $b(\Delta)$. Note that all irreducible subquotients of $\pi_1\times\pi_2$ must have the same supercuspidal support, i.e.~$\supp(L(\gotM_1+\gotM_2))$. We may assume that $\supp(\sigma)=\supp(L(\gotM_1+\gotM_2))$, for otherwise the statement is trivially true. Thus, $b(\Delta)$ is also the minimal point in $\underline{\supp(L(\gotM_1+\gotM_2))}$.

Since $\gotM_1,\gotM_2$ are ladders, $\supp(\sigma)(b(\Delta))\leq2$. Hence, if $m(\sigma_\otimes,\,\mathbf{r}_{\alpha_\sigma} (\pi_1\times\pi_2))>0$, then there can be at most one segment $\hat{\Delta}\in\underline{\gotN-\Delta}$ with $b(\hat{\Delta})=b(\Delta)$. In case there is no such segment, let us still write $\hat{\Delta} = [b(\Delta), b(\Delta)-1]$. Moreover, let us always assume that $e(\hat{\Delta})\leq e(\Delta)$.

 We write $\sigma'= L(\gotN -\Delta-\hat{\Delta})$. Then, $\sigma_\otimes = L(\Delta+\hat{\Delta})\otimes \sigma'_\otimes\in \irr(M_{\alpha_\sigma})$. Let us write $M_\beta\cong G_{m_1}\times G_{m_2}$ for the standard Levi subgroup of $G_n$ corresponding to $L(\Delta+\hat{\Delta})\otimes \sigma'$. Then $M_{\alpha_\sigma}$ can be written as $G_{m_1}\times M_{\alpha_{\sigma'}}$.

By the Geometric Lemma (see again \cite[Section 1.2]{LM2}), we know that
\[
[\mathbf{r}_\beta(\pi_1\times \pi_2)] = \sum_i [\tau_1^i\times \tau_2^i]\otimes [\delta_1^i\times \delta_2^i],
\]
where $i$ goes over all possible irreducible subquotients $\tau_1^i\otimes \delta_1^i$ and  $\tau_2^i\otimes \delta_2^i$ of the appropriate Jacquet modules of $\pi_1$, $\pi_2$, respectively. Since $\mathbf{r}_{\alpha_\sigma} = (1\otimes \mathbf{r}_{\alpha_{\sigma'}})\mathbf{r}_\beta$, we have
\[
m(\sigma_\otimes,\,\mathbf{r}_{\alpha_\sigma} (\pi_1\times\pi_2)) = \sum_i m(L(\Delta+\hat{\Delta}),\,\tau_1^i\times \tau_2^i)\cdot m(\sigma'_\otimes,\,\mathbf{r}_{\alpha_{\sigma'}} (\delta_1^i\times \delta_2^i)).
\]

Recall that from the result of \cite{LapidKret} it follows that $\tau_j^i, \delta_j^i$ are all ladder representations. By the induction hypothesis it is enough to show that $m(L(\Delta+\hat{\Delta}),\,\tau_1^i\times \tau_2^i)=0$ for all $i$, except for possibly one index $i_0$, and that $m(L(\Delta+\hat{\Delta}),\,\tau_1^{i_0}\times \tau_2^{i_0})\leq1$. The latter statement indeed follows from Lemma \ref{mul-lem}.

Suppose that the Jacquet modules of $\pi_1,\pi_2$ have respective irreducible subquotients $\tau_1\otimes \delta_1$,  $\tau_2\otimes \delta_2$, for which $\tau_1\times\tau_2$ contains $L(\Delta+\hat{\Delta})$ as a subquotient. It remains to show the uniqueness of such subquotients.

Note first that if $\Delta=\hat{\Delta}$, then Lemma \ref{mul-lem} implies that $\tau_1\cong\tau_2\cong L(\Delta)$, which gives the wanted conclusion. Thus, we can assume $e(\hat{\Delta})< e(\Delta)$ henceforth.

Lemma \ref{mul-lem} now states that such pair $\{\tau_1,\tau_2\}$ must be of the form $\mathcal{P}=\{\rho_1,\rho_2\}$, where

\[
\rho_1= L\left( \sum_{i=0}^t [a_{2i},a_{2i+1}-1] \right),\quad \rho_2 = L\left( \hat{\Delta}+\sum_{i=1}^s  [a_{2i-1},a_{2i}-1] \right),  
\]
for some $a=a_0,\; e(\hat{\Delta})+1 <a_1<\ldots<a_l=b+1$, with either $l=2t+1=2s+1$ or $l=2s= 2t+2$. Here we write $\Delta= [a,b]$. 

Let us write $\pi_i = L(\Delta^i_1+\ldots+\Delta^i_k)$ with $e(\Delta^i_{k_i})<\ldots<e(\Delta^i_1)$, for $i=1,2$. Recall from \cite{LapidKret} that $\tau_i$, being a leftmost Jacquet module component, must be expressible in the form $L\left( \sum_{j=1}^{k_i} [c^i_j,e(\Delta^i_j)] \right)$, where $c^i_{k_i}<\ldots<c^i_1$ and $b(\Delta^i_j)\leq c^i_j \leq e(\Delta^i_j)+1$ for all $1\leq j\leq k_i$ and $i=1,2$.

Let us first consider the case that $\hat{\Delta}$ is a non-trivial segment. Then, $\supp(L(\gotM_1+\gotM_2))(b(\Delta))=2$ and $c^1_{k_1} = c^2_{k_2}=b(\Delta)$. Comparing the descriptions of $\mathcal{P}$ with that of $\{\tau_1,\tau_2\}$, we see that $\{e(\Delta^1_{k_1}),e(\Delta^2_{k_2})\} = \{e(\hat{\Delta}), a_1-1\}$. Since $e(\hat{\Delta})<a_1-1$, there is a unique identification between $\mathcal{P}$ and $\{\tau_1,\tau_2\}$. Hence, we can assume without loss of generality that $\tau_1=\rho_1$ and $\tau_2=\rho_2$.

Now, in the other case, that is when $\hat{\Delta}$ is the trivial segment, $\supp(L(\gotM_1+\gotM_2))(b(\Delta))=1$. Without loss of generality we can assume that $m(b(\Delta),\,\supp(\pi_1))=1$. Clearly, we again have $\tau_1=\rho_1$ and $\tau_2=\rho_2$.

Assume now there is a different pair $\tau'_1\otimes \delta'_1$,  $\tau'_2\otimes \delta'_2$ of irreducible subquotients of the respective Jacquet modules of $\pi_1,\pi_2$, such that
\[
\tau'_1= L\left( \sum_{i=0}^{t'} [a'_{2i},a'_{2i+1}-1] \right),\quad \tau'_2 = L\left(\hat{\Delta}+ \sum_{i=1}^{s'} [a'_{2i-1},a'_{2i}-1] \right),
\]
with similar assumptions on the indices.

Let $1\leq i_0$ be the minimal index for which $a_{i_0}\neq a'_{i_0}$. We can assume that $a'_{i_0}< a_{i_0}$. Also, without loss of generality we can assume that $i_0$ is odd. Otherwise, $\pi_2$ should replace $\pi_1$ for the rest of the argument. 

Note that $a_{i_0}-1 = e(\Delta^1_{j_0})$ and $a_{i_0-1} = c^1_{j_0}$ for some $j_0$. Using a similar reasoning on $\tau'_1$, we see there is an index $j_1$ for which $a'_{i_0}-1 = e(\Delta^1_{j_1})$. Hence, $e(\Delta^1_{j_1})<e(\Delta^1_{j_0})$ and $j_0<j_1$. Therefore,
\[
c^1_{j_1} < c^1_{j_0} = a_{i_0-1} = a'_{i_0-1} < a'_{i_0} = e(\Delta^1_{j_1})+1.
\]
In particular, $[c^1_{j_1},e(\Delta^1_{j_1})]$ is a non-trivial segment in the multisegment defining $\tau_1$, which means $a'_{i_0} = a_{i_1}$ for some odd $i_1<i_0$. But, since $i_1< i_0-1$ we get a contradiction from $a_{i_1}<a_{i_0-1}<a'_{i_0}$.

\end{proof}

\begin{corollary}[Theorem \ref{main-2}]
Given ladder representations $\pi_1,\pi_2\in \irr$, $m(\sigma,\,\pi_1\times \pi_2)\leq1$ for all $\sigma\in \irr$.
\end{corollary}

\begin{proof}
Suppose $\pi_1\in \irr_{[\rho_1]}$ and $\pi_2\in \irr_{[\rho_2]}$. If $\rho_1\not\in \mathbb{Z}_{[\rho_2]}$ then $\pi_1\times \pi_2$ is irreducible and there is nothing to prove. Therefore we are free to assume $\pi_1,\pi_2\in \irr_{[\rho]}$. 

Now we assume the contrary, that is, there exists $\sigma\in \irr_{[\rho]}$ with $m(\sigma,\,\pi_1\times\pi_2)\geq 2$. Since $m(\sigma_\otimes,\,\mathbf{r}_{\alpha_\sigma} (\sigma))>0$, that would mean that $m(\sigma_\otimes,\,\mathbf{r}_{\alpha_\sigma}(\pi_1\times\pi_2))\geq 2$. This contradicts Proposition \ref{mult-one}.

\end{proof}

\bibliographystyle{abbrv}
\bibliography{propo2}{}

\end{document}